\documentclass[12pt,a4paper]{amsart}
\usepackage{graphicx,amssymb}

\usepackage[all]{xy}

\newtheorem{thm}{Theorem}[section]
\newtheorem{cor}[thm]{Corollary}
\newtheorem{lem}[thm]{Lemma}
\newtheorem{prop}[thm]{Proposition}
\theoremstyle{definition}

\newtheorem{rem}[thm]{Remark}

\numberwithin{equation}{section}

\newcommand{\QQ}{\mathbb Q}

\newcommand{\ZZ}{\mathbb Z}
\newcommand{\CC}{\mathbb C}
\newcommand{\PP}{\mathbb P}

\newcommand{\ra}{\rightarrow}

\DeclareMathOperator{\End}{{End}}
\usepackage{hyperref}

\begin{document}

\title[ ]{Jacobians with complex multiplication}%
 \author{Angel Carocca, Herbert Lange and Rub{\'\i} E. Rodr{\'\i}guez}

\address{A. Carocca\\Facultad de Matem\'aticas,
Pontificia Universidad Cat\'olica de Chile, Casilla 306-22,
Santiago, Chile}
\email{acarocca@mat.puc.cl}

\address{H. Lange\\Mathematisches Institut,
              Universit\"at Erlangen-N\"urnberg\\Germany}
\email{lange@mi.uni-erlangen.de}

\address{R. E. Rodr{\'\i}guez\\Facultad de Matem\'aticas,
Pontificia Universidad Cat\'olica de Chile, Casilla 306-22,
Santiago, Chile}
\email{rubi@mat.puc.cl}
\thanks{The first and third author were supported by Fondecyt
grants 1095165 and  1060742  respectively} 

\begin{abstract}
We construct and study two series of curves whose Jacobians admit
complex multiplication. The curves arise as quotients of Galois
coverings of the projective line with Galois group metacyclic groups
$G_{q,3}$ of order $3q$ with $q \equiv 1 \mod 3$ an odd prime, and
$G_m$ of order $2^{m+1}$. The complex multiplications arise as
quotients of double coset algebras of the Galois groups of these
coverings. We work out the CM-types and show that the Jacobians are
simple abelian varieties.
\end{abstract}

\maketitle

\section{Introduction}

An abelian variety $A$ over an algebraically closed field is said to have
or to admit \textit{complex multiplication} if there is a CM-field $K$ of
degree $2 \dim(A)$ over $\QQ$ such that $K \subset \End_{\QQ}(A)$. A
smooth projective curve $C$ is said to admit \textit{complex
multiplication}, if its Jacobian variety does. In these cases one says
that $A$ (respectively $C$) has complex multiplication \textit{by} $K$. In
this paper we use Galois coverings of the projective line with metacyclic
Galois groups, in order to construct and investigate two series of curves
with complex multiplication.\\

For the first series let $q$ be an odd prime and $n$ a positive
integer such that $n|q-1$. Consider the group
$$
G_{q,n}:= \;\langle a,b\;|\; a^q=b^n=1,\; b^{-1}ab = a^k \rangle
$$
where $ 1 < k < q$ is such that $k^n \equiv 1 \mod q \;$ and $ \; \;
k^m \not\equiv 1 \mod q \;$ for all  $1 \leq m < n  $ (that is, the
order of $k$ mod $q$ is $n$). Denote the subgroup generated by $b$ by
$$
H = \langle b \rangle.
$$
In \cite{e} Ellenberg used Galois coverings $Y$ of the projective line
with Galois group $G_{q,n}$, such that the Jacobian of the quotient curve
$X = Y/H$ admits a totally real field as an endomorphism algebra. In this
note we show that his method can also be applied to construct smooth
projective curves admitting complex multiplication. In fact, we show that
for every $q$ as above with $n=3$ there is exactly one Galois covering $Y$
such that $X = Y/H$ has complex multiplication. To be more precise, our
first result is the following theorem. To state it, let $\zeta = \zeta_q$
denote a primitive $q$-th root of unity and let $\QQ(\zeta^{(n)})$ denote
the unique subfield of index $n$ of the cyclotomic field $\QQ(\zeta)$.
Clearly $\QQ(\zeta^{(n)})$
is a CM-field if and only if $n$ is odd, which we assume in the sequel.\\

\noindent {\bf Theorem 1.} {\it Suppose $Y$ is a Galois covering of
$\PP^1$ with group $G_{q,n}$ with $n$ odd and branch points $p_i \in \PP^1$ of
ramification index $n_i$ for $i=1, \ldots, r$ over an algebraically
closed field $\mathcal{K}$ of characteristic 0.

\begin{enumerate}
\item The curve $X = Y/H$ admits complex multiplication by $\QQ(\zeta^{(n)})$ if and only if $n= r=3$ and
$\{n_1,n_2,n_3\} = \{q,3,3\}$.
\item For every odd prime $q \equiv 1 \mod 3$ and $n=3$ there is, up to isomorphism, exactly one such curve $Y$.

Furthermore, in this case the following results hold.

\item The Jacobian $JY$ is isogenous to $JX^3$.
\item The Jacobian $JX$ is a simple abelian variety, of dimension $\dfrac{q-1}{6}$.
\item The function field of $Y$ over $\mathcal{K}$ is
$$
\mathcal{K}(Y) = \mathcal{K}(z,y)
$$
where $\mathcal{K}(\PP^1) = \mathcal{K}(x)$ and $z$ and $y$ satisfy the equations
$$
z^3 = \frac{x}{x-2} \; \text{ and } \;  y^q =
(z-1)(z-\omega_3)^k(z-\omega_3^2)^{k^2},
$$
with $\omega_3$ a primitive third root of unity.
\end{enumerate}}

\vspace{0.5cm}

For the second series of curves let $m \geq 3$ and consider the
group
$$
G_m = \langle a, b \;|\; a^{2^m}= b^2=1, bab= a^{d} \rangle
$$
where $d = 2^{m-1}-1$ (note that $d^2 \equiv 1$ mod $2^m$).

Let $\xi = \xi_{2^m}$ denote a primitive $2^m$-th-root of unity, and
observe that $\xi + \xi^d = \xi - \overline{\xi}$ is not real.

Consider the complex irreducible representation $V$ of $G_m$ given
by
$$
V(a) = \left(
      \begin{array}{cc}
        \xi & 0 \\
        0 & \xi^d \\
      \end{array}
    \right)
\ , \ V(b) = \left(
      \begin{array}{cc}
        0 & 1 \\
        1 & 0 \\
      \end{array}
    \right)
$$
\\
Its character field  $ \; K_V = \mathbb{Q}[\xi + \xi^d] \; $ is a cyclic
CM-field of degree $ \; [K_V : \mathbb{Q}] = 2^{m-2}$.

\vspace{2mm}

\noindent {\bf Theorem 2}. \textit{Let $ \; m \geq 3. \; $ Then
\begin{enumerate}
\item There exists a Galois covering   $ \; Y \to \mathbb{P}^1 \; $
with Galois group $G_m$, branched at $3$ points in $\PP^1$  with
monodromy $a, \; b  \; \; $ and $ \;  ab \; $.
\item The curve  $X = Y/\langle b\rangle$ and the Prym variety $P$ of the covering $Y \ra X$ have complex multiplication by $K_V$.
\item $JY$ is isogenous to $JX^2$.
\item  $JX$ and $P$ are isogenous simple principally polarized abelian
    varieties, of dimension $2^{m-3}$.
\item $Y$ and $X$ are hyperelliptic curves. An equation for $Y$ is
$$
y^2 = x(x^{2^{m-1}} -1).
$$
\end{enumerate}}

\vspace{0.5cm}

In the second section we fix the notation and collect some
preliminaries on the representations of the groups $G_{q,n}$ and
$G_m$.

Section 3 contains the proof of Theorem 1. To be more precise, in
3.1 we prove parts (1) and (2) of the theorem. In 3.2 we see how the
Jacobians of $X$ and $Y$ are related. In fact, $JY$ is isogenous to
the third power of $JX$ (part (3)). In particular $JY$ also admits
complex multiplication. In 3.3 we apply the theorem of
Chevalley-Weil in order to compute the CM-type of $JY$. This implies
part (4) of the theorem. Finally in 3.4 we work out the function
field of the curve $Y$.

In Section 4 we prove Theorem 2. To be more precise, in 4.1 we prove
(1), (2) and (3) of Theorem 2. Section 4.2 contains the proof of
(4). Moreover, in this section we compute the CM-types of $JY$ and
$JX$. Finally, in 4.3 we give the equations of the curves $Y$ and
$X$. We did not include the equation of $X$ in Theorem 2, because it
requires some more notation.

We would like to thank Eduardo Friedman and Wulf-Dieter Geyer for
some valuable conversations.

\section{Preliminaries}

\subsection{Some notation}
For any finite group $G$ we denote by $\chi_0$ the trivial
representation of $G$. If $H$ is any subgroup of $G$, $\chi_H$ will
denote the character of the representation of $G$ induced by the
trivial representation of $H$. If $H$ is cyclic generated by $g \in
G$, we also write $\chi_g$ for $\chi_{\langle g \rangle}$. In
particular, if $H = \{1\}$, then $\chi_1 = \chi_{\{1\}}$ is the
character of the regular representation of $G$.

If $V$ is a representation of $G$, then $V^H$ denotes the subspace
of $V$ fixed by $H$.

All curves will be smooth, projective and irreducible; for
simplicity we assume the curves to be defined over the field of
complex numbers. As in \cite{e}, the results remain valid over any
algebraically closed field of characteristic not dividing the
group orders by using $l$'adic cohomology and Grothendieck's
algebraic fundamental group instead of singular cohomology and
usual monodromy.

\subsection{Representations of $G_{q,n}$}

Let $G_{q,n}$ denote the group defined in the introduction. As a
semidirect product of the subgroup $N = \langle a \rangle$  by the
subgroup $H = \langle b \rangle$, it is a metacyclic group of order
$qn$.

Let $\omega = \omega_n$ be a primitive $n$-th root of unity. The
non-isomorphic one-dimensional representations of $G_{q,n}$ are the
following:
$$
\chi_i(a) = 1\;; \; \chi_i(b) = \omega^i \; \text{ for } \; i = 0,
\ldots n-1.
$$

There are exactly
$$
s := \frac{q-1}{n}
$$
complex irreducible representations $V_i$ of dimension $n$, defined as
follows. If $a^{i_1}, \ldots, a^{i_s}$ with $s = \frac{q-1}{n}$ is a set
of representatives for the action of $H$ on $N$ defined by the relation
$b^{-1}ab = a^k$, the corresponding orbits are $\{a^{i_j},a^{k i_j},
\ldots, a^{k^{n-1}i_j}\}$ for $j=1, \ldots,s$. If $\zeta = \zeta_q$
denotes a primitive $q$-th root of unity, then for $j = 1, \ldots, s$ the
representation $V_j$ is given by
$$
V_j(a) = \left( \begin{array}{ccccc}
                             \zeta^{i_j}&0&0& \cdots & 0\\
                             0& \zeta^{ki_j}& 0& \cdots & 0\\
                             \vdots & \vdots& \ddots & \cdots& \vdots\\
                             0& 0 & 0 & \cdots & \zeta^{k^{n-1}i_j}
                             \end{array} \right); \quad  \quad  V_j(b) =  \left( \begin{array}{cccccc}
                             0&0&0& \cdots & 0 & 1\\
                             1& 0& 0 & \cdots & 0 & 0\\
                             0&1&0& \cdots & 0 & 0\\
                             \vdots & \vdots & \ddots & \cdots & \vdots & \vdots \\
                             0 & 0 & 0 & \cdots & 1 & 0
                             \end{array} \right).
$$
For any $\ell|n$ consider the subgroup $S_{\ell} = \langle
b^{\frac{n}{\ell}} \rangle$ of $H$ of order $\ell$.  Also, $S_{\ell} \subset
\ker(\chi_i)$ if and only if $n | \frac{n}{\ell} i$, and there are exactly
$\frac{n}{\ell}$ such $i$'s. We have
$$
\dim V_j^H = 1 \; \text{ and } \;  \dim V_j^{S_{\ell}} = \frac{n}{\ell} \, ,
$$
and it is easy to check that
\begin{align*}
\chi_N  & = \chi_0 + \chi_1 + \cdots + \chi_{n-1},\\
\chi_H & = \chi_0 + \chi_{V_1} + \chi_{V_2} + \cdots + \chi_{V_s}, \\
\chi_{S_{\ell}} & = \sum_{0 \leq i < n \atop n| \frac{n}{\ell} i} \chi_i +
\frac{n}{\ell}(\chi_{V_1} + \chi_{V_2} + \cdots + \chi_{V_s}).
\end{align*}

We will use the following result, for the proof of which we refer to
\cite{br} or \cite{e}. Let $G$ denote a finite group acting on a
compact Riemann surface $Y$ with monodromy $g_1, g_2, \ldots, g_u$.
Let $\chi_Y$ denote the character of the representation
$H^1(Y,\QQ)$. Then
\begin{equation} \label{eq2.1}
\chi_Y = 2 \chi_0 + (2g(Y/G) -2 + u) \chi_{\{1 \}} - \sum_{j=1}^u \chi_{\langle g_j \rangle}.
\end{equation}

The following proposition appears essentially in \cite{e}.
\begin{prop} \label{prop2.1}
Suppose the metacyclic group $G_{q,n}$ acts on the compact Riemann
surface $Y$ with monodromy $g_1, \ldots, g_r, g_{r+1} \ldots,
g_{r+t}$, where $g_{j}$ has order $n_{j}$ (dividing $n$) for $j = 1,
\ldots, r$ and order $q$ for $j = r+1 \ldots,u= r+t$ and assume
$g(Y/G_{q,n}) = 0$. Then
$$
\chi_Y = (r-2) \sum_{i=1}^{n-1} \chi_i - \sum_{j = 1}^r \sum_{0<i<n \atop n| \frac{n}{n_j} i} \chi_i
+ \left( n(r+t-2) - \sum_{j=1}^r \frac{n}{n_j} \right) ( \chi_{V_1} + \cdots \chi_{V_s}).
$$
\end{prop}
\begin{proof}
First observe that if $g_j$ has order $n_j$ dividing $n$, then $\langle g_j \rangle$ is conjugate to $S_{n_j}$,
and that if $g_j$ has order $q$ then  $\langle g_j \rangle = N$.

Then, from \eqref{eq2.1} and using the formulas above for $\chi_N, \chi_H$ and
$\chi_{S_{\ell}}$, we obtain
\begin{eqnarray*}
\chi_Y & = & 2 \chi_0 + (r+t-2)\left( \chi_0 + \cdots + \chi_{n-1} + n(\chi_{V_1} + \cdots + \chi_{V_s}) \right)\\
&& - \left( t(\chi_0 + \cdots + \chi_{n-1}) + (\sum_{j=1}^r \frac{n}{n_j} ) ( \chi_{V_1} + \cdots + \chi_{V_s}) +
\sum_{j = 1}^r \sum_{0 \leq i < n \atop n| \frac{n}{n_j} i} \chi_i \right).
\end{eqnarray*}
This implies the assertion.
\end{proof}

\subsection{Representations of the group $G_m$}

For $m \geq 3$ let $G_m$ denote the group defined in the introduction. As
a semidirect product of the subgroup $N = \langle a \rangle$  by the
subgroup $H = \langle b \rangle$, it is a metacyclic group of order
$2^{m+1}$.

The group $G_m$ has $ \, 3 \, $  nontrivial representations of degree $ \,
1$, namely
$$
\chi_1: \; a \mapsto 1, \;  b \mapsto  -1, \qquad
\chi_2: \; a \mapsto -1, \;  b \mapsto  1,  \qquad
\chi_3: \; a \mapsto -1, \;  b \mapsto  -1.
$$
Let $\xi = \xi_{2^m}$ denote a primitive $2^m$-th root of unity. For
$i = 1, \ldots, m-1$ consider the representation $V_i$ defined by
$$
V_i(a) = \left(
      \begin{array}{cc}
        \xi^{2^{i-1}} & 0 \\
        0 & \xi^{2^{i-1}d} \\
      \end{array}
    \right),
\qquad \ V_i(b) = \left(
      \begin{array}{cc}
        0 & 1 \\
        1 & 0 \\
      \end{array}
    \right).
$$
Note that $V_1$ is the representation $V$ of the introduction. $V_i$
is a complex irreducible representation with Galois character field
$$
 K_i = \mathbb{Q}[\xi^{2^{i-1}} + \xi^{2^{i-1}d}]
$$
with $ \; [K_i : \mathbb{Q}] = 2^{m-1-i}$. Hence, $ \; V_i \; $ has
$ \; 2^{m-1-i} \; $ non-equivalent complex irreducible
Galois-conjugate representations $V_i^1, \ldots, V_i^{2^{m-1-i}}$.
Since these representations are obviously pairwise non-equivalent,
we get in this way
$$
2^{\: m-2} + 2^{\: m-3} + ... + 1 = 2^ {\:m - 1} - 1
$$
complex irreducible representations of degree $2$. Now
$$
4 \cdot 1^2 + ( 2^ {\:m - 1} -1)\cdot 2^{\: 2} = 2^{\: m + 1} = \vert \: G_m \: \vert,
$$
which implies that these are all the complex irreducible representations
of $G_m$.

The non-equivalent rational irreducible representations are, apart from
$\chi_0, \ldots, \chi_3$, the representations $W_i$, $i=1, \ldots, m-1$,
whose complexifications are
$$
W_i \otimes_{\QQ} \CC \simeq \oplus_{j=1}^{2^{m-1-i}} V_i^j \, .
$$
Note that $W_i$ is of degree $2^{m-i}$.

One checks that
\begin{align*}
\chi_N & = \chi_0 + \chi_1, \\
\chi_H & = \chi_0 + \chi_2 + \sum_{i=1}^{m-1} \chi_{W_i}, \\
\chi_{\langle ab \rangle} & = \chi_0 + \chi_3 + \sum_{i=2}^{m-1}
\chi_{W_i}.
\end{align*}

Using this, we immediately obtain from \eqref{eq2.1},
\begin{prop} \label{prop2.2}
Let $Y \ra \PP^1$ denote a Galois covering with Galois group $G_m$, $m
\geq 3$, ramified over $3$ points of $\PP^1$ with monodromy $a$, $b$ and
$(ab)^{-1}$. Then
$$
\chi_Y = \chi_{W_1}.
$$
\end{prop}
\vspace{0.3cm}

\section{Curves with Galois group $G_{q,n}$}

\subsection{Curves with complex multiplication by $\QQ(\zeta_q^{(n)})$}
Let $Y \ra \PP^1$ be a Galois covering with group $G_{q,n}$ over an
algebraically closed field $\mathcal{K}$ of characteristic $0$. Consider the
curve
$$
X := Y/H
$$
where $H$ denotes the subgroup generated by $b$. In this section we
use Proposition \ref{prop2.1} in order to determine those Galois
coverings $Y$ for which the curve $X$ admits complex multiplication.

Recall that a \textit{CM-field} $K$ of degree $2g$ is a totally
complex quadratic extension of a totally real field of rank $g$ over
$\QQ$.

The main result of this section is the following Proposition. As in the
introduction let $\QQ(\zeta_q^{(n)})$ denote the unique subfield of index
$n$ of the cyclotomic field $\QQ(\zeta_q)$. It is a CM-field if and only
if $n$ is odd, which we assume in the sequel.

\begin{prop}  \label{prop3.1}
Suppose $Y$ is a Galois covering of $\PP^1$ with group $G_{q,n}$ and
branch points $p_i \in \PP^1$ of ramification index $n_j$ dividing $n$ for
$i=1, \ldots, r$ and equal to $q$ for $j = r+1, \ldots, r+t$.

Then the curve $X = Y/H$ has complex multiplication by
$\QQ(\zeta_q^{(n)})$ if and only if $n=r+t=3$ and $\{n_1,n_2,n_3\} =
\{3,3,q\}$.
\end{prop}

\begin{proof}
 According to the Hurwitz formula, the  genus of $ \; X $ is
$$
g(X) = \frac{q-1}{2} \left( r+t-2-\sum_{j = 1}^{r} \frac{1}{n_j} \right).
$$
If the curve $X$ has complex multiplication by $\QQ(\zeta_q^{(n)}) \; $
then $[\QQ(\zeta_q^{(n)}): \QQ] = 2\dim(JX)$. This is the case if and only
if
$$
\frac{q-1}{n} = 2 \, \frac{q-1}{2} \left(  r+t-2-\sum_{j = 1}^{r}
\frac{1}{n_j} \right)
$$
which is equivalent to
\begin{equation}  \label{eq3.2}
\sum_{j=1}^r \frac{n}{n_{j}} = n(r+t-2) - 1.
\end{equation}
This implies $t \leq 2$.

If $t = 2$, then \eqref{eq3.2} says $\sum_{j=1}^r
\displaystyle\frac{\:n}{n_j} = nr -1$, a contradiction, since $\;
\displaystyle\frac{n}{n_j}$ is either 1 or $\geq 3$.

If $t = 0$, then \eqref{eq3.2} says $\sum_{j=1}^r
\displaystyle\frac{\:n}{n_j} = n(r-2) -1$. Since $n$ is odd, $n_j
\geq 3$, which implies $\; \displaystyle\frac{nr}{3} \geq nr -2n
-1$ and thus $r \leq 3 + \displaystyle\frac{3}{2n}$. This gives $r
= 2$ or 3. In both cases \eqref{eq3.2} cannot be satisfied. For
$r=2$ the right hand side would be negative and for $r=3$ the even
number $n-1$ is not the sum of 3 genuine divisors of $n$.

Hence $t=1$, in which case \eqref{eq3.2} says $\sum_{j=1}^r
\displaystyle\frac{n}{n_j} = n(r-1) -1$. By the same argument as
above, this implies $\; \displaystyle\frac{nr}{3} \geq nr -n -1$,
which gives $\; \displaystyle\frac{2}{3}r \leq 1 +
\displaystyle\frac{1}{n}$. Since $n \geq 3$, this implies $r \leq
\displaystyle\frac{3}{2}(1 + \displaystyle\frac{1}{3}) = 2$ and
thus $r=2$. But then \eqref{eq3.2} says
$$
\frac{n}{n_1} + \displaystyle\frac{n}{n_2} = n-1
$$
whose only odd solution is $n=n_1=n_2=3$.

Hence we have shown that the curve $X = Y/H$ has  complex
multiplication by $\QQ(\zeta_q^{(n)})$ only if $n = r + t = 3 \; $
and $ \; \{n_1,n_2,n_3\} = \{3,3,q\}$.

Certainly there exists a Galois covering of this type, with
branch points $p_1, p_2, p_3$ in $\PP^1$ and stabilizers $G_{p_1} =
\;\langle a \rangle,\; G_{p_2} = \;\langle b \rangle$ and $G_{p_3} =
\langle ab \rangle$, since $ab(ab)^{-1}= 1$.\vspace{2mm}\\
We will now finish the proof by showing that under the assumptions
$\QQ(\zeta_q^{(3)})$ is contained in $\End_{\QQ}(JX)$.\\

Let $W$ be the rational irreducible representation of $G_{q,3}$ whose
complexification is the direct sum of all the irreducible complex
representations $V_j$ of dimension $3$.

Note that Proposition \ref{prop2.1} in our case says $\chi_Y = \chi_W$.
This, together with  \cite[p. 202, Corollaire]{g}, implies that there is
an isomorphism of $\QQ[G]$-modules
$$
H^1(Y, \QQ) \simeq W
$$
and an isomorphism of $\QQ[H \backslash G/H]$-modules
\begin{equation} \label{eq2.2}
H^1(X,\QQ) \simeq H^1(Y, \QQ)^H \simeq (W^H)^{\oplus m}
\end{equation}
with $ \; m = n(r+t-2) - \displaystyle\sum_{j=1}^r
\displaystyle\frac{n}{n_j} = 3( 2 + 1 - 2) -
\displaystyle\sum_{j=1}^2 1  = 1$. \\

Moreover, \eqref{eq2.2} implies that the canonical homomorphism $\QQ[H
\backslash G/H] \ra \End_{\QQ}(JX)$ induces a homomorphism $\rho: \QQ[H
\backslash G/H] \ra \End(H^1(X, \QQ)) = \End(W^H)$, and the image of
$\rho$ is isomorphic to the image of $\QQ[H \backslash G/H]$ in
$\End_{\QQ}(JX)$.

Now the image of $\QQ[G]$ in $\End(W)$ is isomorphic to the $3 \times
3$-matrix algebra over $\QQ(\zeta_q^{(3)})$, and the fact that $\dim V_1^H
=1$ implies that $ \; \QQ[H \backslash G/H] \: \cong \: \QQ \oplus
\QQ(\zeta_q^{(3)})$ (see \cite[Theorem 4.4]{cr} and \cite{e}).

Hence the image of $\QQ[H \backslash G/H]$ in $\End(W^H)$ is isomorphic to
$\QQ(\zeta_q^{(3)})$ and therefore
$$
\QQ(\zeta_q^{(3)}) \subset \End_{\QQ}(JX).
$$
\end{proof}

We have thus proven the following result.

\begin{cor} \label{cor2.2}
Let $Y$ be a Galois covering of $\PP^1$ with group $G := G_{q,3}$.
Then $X = Y/H$ has complex multiplication by $\QQ(\zeta_q^{(3)})$ if
and only if $Y \ra \PP^1$ is isomorphic to the $G$-covering with
branch points $p_1, p_2, p_3$ in $\PP^1$ and stabilizers $G_{p_1} =
\;\langle a \rangle,\; G_{p_2} = \;\langle b \rangle$ and $G_{p_3} =
\langle ab \rangle$.
\end{cor}

\begin{rem}
For each $q$ there exists a unique curve $Y$ satisfying the conclusions
in the Corollary; hence for every $q$ there is only one such curve $X$, as was observed already by Lefschetz in \cite[p. 463]{l}.
\end{rem}

\subsection{The Jacobians $JY$ and $JX$}

Let $q$ be an odd prime with $q \equiv 1 \mod 3$ and denote
$$
G := G_{q,3} = \; < a,b\;|\; a^q = b^3 = 1, \; b^{-1}ab = a^k >
$$
where $ k^3 \equiv 1 \mod q,\; 1 < k < q$. Let $Y$ and $X$ denote
the curves of Corollary \ref{cor2.2}. In this section we want to see
how the Jacobians $JY$ and $JX$ are related.

First, the Hurwitz formula gives
\begin{equation} \label{eq3.1}
g(Y) = \frac{q-1}{2} \; \text{ and } \;  g(X) = \frac{q-1}{6}.
\end{equation}
In particular $g(Y) = 3 g(X)$.
The following proposition is more precise.
\begin{prop} \label{thm3.1}
The Jacobian of $Y$ is isogenous to the third power of the Jacobian of X:
$$
JY \sim JX^3.
$$
In particular $JY$ admits complex multiplication by $\QQ(\zeta_q)$.
\end{prop}

\begin{proof}
There are exactly $2$ nontrivial rational irreducible
representations of $G$, namely $W_1$, whose complexification is
$\chi_1 \oplus \chi_2$, and $W_2$, whose complexification is $V_1
\oplus \cdots \oplus V_s$, with $\chi_i$ and $V_j$ as defined in
Section 2. Correspondingly, according to \cite[Proposition 5.2]{cr},
there are abelian subvarieties $B_1$ and $B_2$ of $JY$, uniquely
determined up to isogeny, such that
$$
JY \sim B_1^{\frac{\dim \chi_1}{m_1}} \times B_2^{\frac{\dim V_1}{m_2}}
$$
and
$$
JX \sim B_1^{\frac{\dim \chi_1^H}{m_1}} \times B_2^{\frac{\dim
V_1^H}{m_2}},
$$
where $m_i$ is the Schur index of the corresponding representation.
Hence $m_1 = m_2 = 1$. Since $\dim V_1 = 3$ and $\dim V_1^H = 1$, it
suffices to show for the first assertion that $\dim B_1 = 0$. This
is a consequence of \cite[Theorem 5.12]{r}, which in our case says
$$
\dim B_1 = [K_{\chi_1} : \QQ](-\dim \chi_1 + \frac{1}{2}\sum_{j=1}^3 (\dim
\chi_1 - \dim \chi_1^{G_j})),
$$
where $K_{\chi_1}$ denotes the field generated by the values of the
character $\chi_1$ over $\QQ$, $G_1 = \langle a \rangle$, $G_2 = \langle b
\rangle$, and $G_3 = \langle ab \rangle$. Hence we get
$$
\dim B_1 = 2(-1 + \frac{1}{2}(1-1+1-0+1-0)) = 0 \, ,
$$
which completes the proof of the first assertion. According to Corollary
\ref{cor2.2}, $JX$ admits complex multiplication by $\QQ(\zeta_q^{(3)})$.
The last assertion follows from the first and the fact that $\QQ(\zeta_q)$
is a degree-three CM-extension of $\QQ(\zeta_q^{(3)})$.
\end{proof}

\begin{rem}
The Prym variety $P$ of the threefold covering $Y \ra X$ is defined as the
connected component containing $0$ of the kernel of the norm map $JY \ra
JX$. Since $JY$ is isogenous to the product $JX \times P$, we deduce from
Corollary \ref{cor2.2} and Proposition \ref{thm3.1} that
$$
P \sim JX^2.
$$
In particular, $P$ admits complex multiplication by a CM-extension of
degree $2$ of $\QQ(\zeta_q^{(3)})$.
\end{rem}

\subsection{The CM-types of $JY$ and $JX$}

Recall that a CM-field $K$ of degree $2g$ admits $g$ pairs of complex
conjugate embeddings into the field of complex numbers. A
\textit{CM-type} of $K$ is by definition the choice of a set of
representatives of these pairs; that is, a set of $g$ pairwise
non-isomorphic embeddings $K \hookrightarrow \CC$.

The Jacobian $JC$ of a curve $C$ admits complex multiplication by a
CM-field $K$, if and only if $H^1(C, \QQ)$ is a $K$-vector space of
dimension $1$. In this case the Hodge decomposition
$$
H^1(C,\CC) = H^0(C,\omega_C) \oplus \overline{H^0(C,\omega_C)}
$$
induces a CM-type on the field $K$. It is called the
\textit{CM-type} of the Jacobian $JC$. In this section we compute
the CM-types of the Jacobians $JY$ and $JX$ of Section 4.

We need some elementary preliminaries. The congruence
$$
k^3 \equiv 1 \mod q
$$
admits exactly $2$ integer solutions with $2 \leq k \leq q-2$. If $k$ is
one such solution, $q-1-k$ is the other one.

For any integer $n$ let $[n]$ denote the uniquely determined integer with
$$
0 \leq [n] \leq q-1 \; \text{ and } \;  [n] \equiv n \mod q
$$
and for any integer $\ell, \; 1 \leq \ell \leq q-1$ consider the set
$$
O_{\ell} := \{ \ell, [k \ell], [k^2 \ell] \}.
$$
\begin{lem} \label{lem4.1}
Let $k$ be an integer with $2 \leq k \leq \frac{q-1}{2}$ and $k^3
\equiv 1 \mod q$. Then

\noindent {\em (1)}
$$
1 + k + [k^2] = q;
$$
{\em (2)} For any $\ell,\; 1 \leq \ell \leq q-1$,
$$
\ell + [k\ell] + [k^2 \ell] = \left\{ \begin{array}{cl}
                                     q& \text{if $2$ of the numbers in $O_{\ell}$ are smaller than } \;  \frac{q-1}{2},\\
                                     2q & \text{otherwise}.
                                     \end{array}  \right.
$$
\end{lem}

\begin{proof}
(1): We have $1 + k + [k^2] \equiv 0 \mod q$. On the other hand,
$1 < 1 + k + [k^2] \leq 1 + \frac{q-1}{2} + q-1 = \frac{3}{2}q -\frac{1}{2}$. This implies the assertion.\\
(2): We have $\ell +k\ell + k^2 \ell = \ell(1 +k + k^2) \equiv 0
\mod q$. Hence $q$ divides $\ell + [k \ell] + [k^2 \ell]$. On the
other hand, $1 \leq \ell + [k\ell] + [k^2 \ell] < 3q$. This implies
$$
\ell + [k\ell] + [k^2 \ell] = q \; \;\text{or} \;\; 2q.
$$
Suppose $2$ of the $3$ numbers are less than $\frac{q-1}{2}$. Then
$\ell + [k\ell] + [k^2 \ell] < 2\frac{q-1}{2} + q-1 < 2q$ and hence
$\ell + [k\ell] + [k^2 \ell] = q$.

In the remaining case at least one of the numbers is $>
\frac{q-1}{2}$, implying $\ell + [k\ell] + [k^2 \ell] > 1 + 2
\frac{q-1}{2} = q$ which completes the proof.
\end{proof}
The following lemma gives a criterion for the set $O_{\ell}$ to
contain two elements less than $\frac{q-1}{2}$.
\begin{lem} \label{lem3.8}
For any integer $\ell;\; 1 \leq \ell \leq q-1$ consider the real
number
$$
\beta_{\ell} := \sin \left(\frac{2\pi \ell}{q} \right) + \sin \left(\frac{2\pi k\ell}{q} \right) +
\sin \left(\frac{2\pi k^2\ell}{q} \right).
$$
Then the set $O_{\ell}$ contains two elements less than
$\frac{q-1}{2}$ if and only if $\beta_{\ell} > 0$.
\end{lem}
\begin{proof}
Suppose $O_{\ell}$ contains two elements less than $\frac{q-1}{2}$.
Without loss of generality we may assume that they are $\ell$ and
$[k\ell]$. Since $1 + k + k^2 \equiv 0 \mod q$, we have
\begin{eqnarray*}
\sin \left( \frac{2\pi[k^2\ell]}{q}\right) & = & \sin \left( \frac{2\pi[-(1+k)\ell]}{q}\right) = - \sin \left( \frac{2\pi[(1+k) \ell]}{q}\right) \\
& = & -\left[ \sin \left( \frac{2\pi \ell}{q}\right) \cos \left( \frac{2\pi[k \ell]}{q}\right) +
\cos \left( \frac{2\pi \ell}{q}\right) \sin \left( \frac{2\pi[k\ell]}{q}\right) \right].
\end{eqnarray*}
Hence
\begin{eqnarray*}
\beta_{\ell} & = & \sin \left(\frac{2\pi \ell}{q} \right) + \sin \left(\frac{2\pi [k\ell]}{q} \right) +
\sin \left(\frac{2\pi [k^2\ell]}{q} \right)\\
& = & \sin \left( \frac{2\pi \ell}{q}\right) \left[ 1 - \cos\left( \frac{2\pi[k \ell]}{q}\right)\right] + \sin \left( \frac{2\pi[k\ell]}{q}\right) \left[ 1 - \cos \left( \frac{2\pi \ell}{q} \right) \right]
\end{eqnarray*}
is positive, because both $\ell$ and $[k \ell]$ are positive and
smaller than $\frac{q-1}{2}$.

The \lq\lq only if\rq\rq $\,$ part of the assertion follows from the
fact that the elements of $O_{q - \ell}$ are the negatives$\! \mod
q$ of the numbers of $O_{\ell}$, and therefore $\beta_{q-\ell} = -
\beta_{\ell}$. This completes the proof.
\end{proof}

Let the notation be as in Section 2 with $n=3$. In particular $\{a^{i_1},
\ldots, a^{i_s}\}\; (s= \frac{q-1}{3})$ denotes a set of representatives
for the action of the group $H = \langle b \rangle$ on the group $N =
\langle a \rangle$. The corresponding orbits are $\{a^{i_j}, a^{ki_j},
a^{k^2i_j}\}$ for $j= 1, \ldots, s$. Now with $\{a^{i_j}, a^{ki_j},
a^{k^2i_j}\}$ also $\{a^{q-i_j}, a^{k(q-i_j)}, a^{k^2(q-i_j)}\}$ is an
orbit, disjoint from it. Hence we can enumerate the orbits in
the following way: The orbits of  $N$ under the adjoint action of the
group $H$ are exactly
\begin{equation} \label{eq4.1}
\{a^{i_{\nu}},a^{ki_{\nu}},a^{k^2i_{\nu}}\}_{\nu = 1}^{\frac{s}{2}}
\; \text{ and } \;
\{a^{q-i_{\nu}},a^{k(q-i_{\nu})},a^{k^2(q-i_{\nu})}\}_{\nu =
1}^{\frac{s}{2}} ,
\end{equation}
where $2$ of the numbers $i_{\nu},[ki_{\nu}],[k^2i_{\nu}]$ are less
than $\frac{q-1}{2}$ (and thus $2$ of the numbers
$q-i_{\nu},[k(q-i_{\nu})],[k^2(q-i_{\nu})]$ are $\geq
\frac{q-1}{2}$).

If $V_j$ denotes the complex irreducible representations as in
Section 2 for $j = 1, \ldots, s$, then we have

\begin{prop} \label{thm4.2}
$$
H^0(Y,\omega_Y) = \bigoplus_{\nu=1}^{\frac{s}{2}} V_{\nu} \, .
$$
\end{prop}
\begin{proof}
Since $3\dfrac{s}{2} = \dfrac{q-1}{2} = g(Y)$, it suffices to show
that every $V_{\nu}$ with $1 \leq \nu \leq \frac{s}{2}$ occurs
exactly once in the representation $H^0(Y,\omega_Y)$ of $G$.

According to the Theorem of Chevalley-Weil (see \cite{cw}), the
representation $V_{\nu}$ occurs exactly
\begin{equation} \label{eqn3.5}
N = -\deg V_{\nu} + \sum_{\mu = 1}^3 \sum_{\alpha= 0}^{n_{\mu - 1}}
N_{\mu,\alpha} \left\langle -\frac{\alpha}{n_{\mu}}\right\rangle
\end{equation}
times in the representation $H^0(Y,\omega_Y)$, where
\begin{itemize}
\item $\mu$ runs through the branch points of the covering $Y \ra \PP^1$,
\item $n_{\mu}$ is the order of the $\mu$-th branch point,
\item $N_{\mu,\alpha}$ denotes the multiplicity of the eigenvalue
    $e^{\frac{2 \pi i \alpha}{n_{\mu}}}$ in the matrix
    $V_{\nu}(g_{\mu})$, where $g_{\mu}$ is any nontrivial element of
    $G$ stabilizing a point in the fiber of $\mu$, and
\item $\langle r\rangle := r - \lfloor r \rfloor$ denotes the fractional part of the real number $r$.
\end{itemize}
Hence we have $n_1 = q,\; n_2=n_3 = 3$ and thus
\begin{eqnarray*}
N &= &-3 + \left\langle -\frac{i_{\nu}}{q}\right\rangle +
\left\langle -\frac{[ki_{\nu}]}{q}\right\rangle +
\left\langle -\frac{[k^2i_{\nu}]}{q}\right\rangle +
2\left\langle -\frac{1}{3}\right\rangle + 2\left\langle -\frac{2}{3}\right\rangle\\
&=& -1 + \frac{q-i_{\nu}}{q} + \frac{q-[ki_{\nu}]}{q} + \frac{q-[k^2i_{\nu}]}{q}\\
&=& 1,
\end{eqnarray*}
where the last equation follows from Lemma \ref{lem4.1} (2), since
by assumption, $2$ of the numbers $q-i_{\nu}, q-[ki_{\nu}],
q-[k^2i_{\nu}]$ are $\geq \frac{q-1}{2}$.
\end{proof}

\begin{cor} \label{cor4.3}
Let $Y$ and $X$ denote the curves of Section 3, and denote $\zeta =
\zeta_q = e^{\frac{2\pi i}{q}}$. Let
$\{a^{i_{\nu}},a^{ki_{\nu}},a^{k^2i_{\nu}}\}_{\nu =
1}^{\frac{s}{2}}$ be the first half of the orbits in
\eqref{eq4.1}. Then

\noindent {\em (1)} The CM-type of $JY$ is given by the following $g(Y) = \dfrac{3s}{2}$
embeddings $\varphi_j$ of $\mathbb{Q}(\zeta)$ into $\mathbb{C}$: $\varphi_j(\zeta) = \zeta^j$ for
$j$ in $\{ i_1 , k \, i_1 , k^2 \, i_1 , \ldots , i_{\frac{s}{2}} , k \, i_{\frac{s}{2}} , k^2 \, i_{\frac{s}{2}} \}$.
\\
{\em (2)} Denoting $\alpha_{{\nu}} := \zeta^{i_{\nu}} +
\zeta^{ki_{\nu}} + \zeta^{k^2i_{\nu}}$ for $\nu = 1, \ldots,
\frac{s}{2}$, the CM-type of $JX$ is given by the following $g(X) = \dfrac{s}{2}$
embeddings $\psi_{\nu}$ of $\mathbb{Q}(\alpha_1)$ into $\mathbb{C}$: $\psi_{\nu}(\alpha_1) = \alpha_{{\nu}}$.
\end{cor}
\begin{proof}
(1) is a direct consequence of Propositions \ref{thm3.1} and
\ref{thm4.2}, and the definition of the representations $V_{\nu}$ in
Section 2.

According to Corollary \ref{cor2.2} the Jacobian $JX$ has complex
multiplication by $\QQ(\zeta^{(3)})$. Hence (2) follows from Theorem
\ref{thm4.2} and the fact that
$$
\QQ(\zeta^{(3)}) = \QQ(\alpha_1) = \QQ(\zeta + \zeta^k + \zeta^{k^2})
$$
is the only subfield of $\QQ(\zeta)$ of index 3.
\end{proof}

\begin{prop} \label{cor4.4}
The Jacobian $JX$  is a simple abelian variety, of dimension
$\frac{q-1}{6}$.
\end{prop}

This proves part (4) of Theorem 1 in the Introduction.

\begin{proof}
Let $\Phi$ denote the CM-type of $JX$, i.e. of the field
$\QQ(\zeta^{(3)})$ as given in Corollary \ref{cor4.3} (2). With
$\alpha = \zeta + \zeta^k + \zeta^{k^2}$ as above and $\mu := \alpha
- \overline{\alpha}$ we have
\begin{itemize}
\item $K_0 = \QQ(\alpha + \overline{\alpha})$ is totally real;
\item $\; \eta := - \mu^2 = 4 \beta_1$ is a totally positive element of $K_0$
(by Lemma \ref{lem3.8});
\item The elements of $\Phi$ are exactly the embeddings $\varphi:
    \QQ(\mu) \hookrightarrow \CC$ for which the imaginary part of
    $\varphi(\mu)$ is positive (according to Lemma \ref{lem3.8}).
\end{itemize}

We have to show that $\Phi$ is a primitive CM-type. For this we
apply the criterion \cite[Prop.27]{st} of Shimura-Taniyama which says the following:
$\Phi$ is primitive if and only if the following two conditions are satisfied:\\
(i) $K_0(\mu) = \QQ(\mu)$;\\
(ii) for any conjugate $\alpha'$ of $\alpha$ over $\QQ$, other than $\alpha$ itself, $\frac{\alpha'}{\alpha}$
is not totally positive.

The first condition holds trivially, since both fields are equal to
$\QQ(\zeta^{(3)})= \QQ(\alpha)$. As for the second condition: for any
conjugate $\mu'$ of $\mu$ over $\QQ$ different from $\mu$, $\;
\frac{\mu'}{\mu}$ is not totally positive, because $\frac{\mu'}{\mu}$ runs
over $\frac{\beta_{\ell}}{\beta_1}$.
\end{proof}

\subsection{The function field $\mathcal{K}(Y)$}

In this section we use Kummer theory in order to prove part (5) of Theorem
1 in the Introduction.

Let $Y$ be the Galois covering of $\PP^1$ of Corollary \ref{cor2.2},
with Galois group $G_{q,3}$ ramified over the points $p_1,\; p_2$
and $p_3$ of $\PP^1$ in affine coordinates. The subgroup $N =
\langle a \rangle$ is normal of index $3$ in $G_{q,3}$, and gives a
factorization of the covering $Y \ra \PP^1$ into cyclic coverings $Y
\ra Z := Y/N$ of degree $q$ and $Z \ra \PP^1$ of degree $3$. The
last covering is ramified over $p_2$ and $p_3$. Hence, according to the
Hurwitz formula,
$$
g(Z) = 0.
$$
We choose an affine coordinate $x$ of $\PP^1$ in such a way that
$p_1 =1, p_2 = 0$ and $p_3 = 2$. Then the covering $Z \ra \PP^1$ is
given by the equation
$$
z^3 = \frac{x}{x-2}
$$
and the function field of $Z = \PP^1$ is $\mathcal{K}(z)$.

\begin{prop}
Let $\omega_3$ denote a primitive third root of unity, and choose $1
< k < q$ such that $k^3 \equiv 1 \mod q$. A (singular) model of the
curve $Y$ is given by the equation
$$
y^q = (z-1)(z-\omega_3)^k(z-\omega_3^2)^{k^2}.
$$
With these notations,  automorphisms $\sigma$ and $\tau$ of the curve $Y$
of corresponding orders $q$ and $3$,  are  given by
$$
\sigma: \; z \mapsto z, \quad y \mapsto \zeta_q y \, , \; \text{
and}
$$
$$
\tau: \;  z \mapsto \omega_3 z, \quad y \mapsto
\frac{\omega_3^{m'}}{(z-\omega_3^2)^m} y^k \,
$$
where $m$ and $m'$ are given by $k^3 = mq+1$ and $k^2 + k + 1 =
m'q$.
\end{prop}

An immediate consequence of the proposition is statement (5) of Theorem 1
in the Introduction.

\begin{proof}
The covering $Y \ra Z$ is ramified exactly over the points $1,
\omega$ and $\omega^2$ of $Z = \PP^1$. According to Kummer theory
the covering $Y \ra Z$ is given by the affine equation
\begin{equation} \label{eq5.1}
y^q = (z-1)(z- \omega_3)^{k}(z- \omega_3^2)^{k^2},
\end{equation}
with $k$ as in the statement of the proposition.

The automorphism $\tau: Z \ra Z$ given by $z \mapsto \omega_3 z$ extends to
the automorphism of $Y$ as defined in the proposition since, denoting the
right hand side of \eqref{eq5.1} by $F = F(z)$, we have
$$
\tau(F) = \omega_3^{1+k + k^2}(z-\omega_3^2)(z- 1)^{k}(z- \omega_3)^{k^2}
= \frac{\omega_3^{1+k+k^2}}{(z-\omega_3^2)^{mq}} F^k
= \left( \frac{\omega_3^{m'}}{(z - \omega_3^2)^m} y^k \right)^q.
$$
If $\sigma$ denotes the automorphism of $Y$ given above, then it is clear
that $Y \ra \PP^1$ is a Galois covering with Galois group $< \sigma, \tau
>\; = G_{q,3}$.
\end{proof}

\section{Curves with Galois group $G_m$}

\subsection{Complex multiplication of the curves}

Let $Y \ra \PP^1$ denote a Galois covering with Galois group
$$
G_m = \langle a, b \;|\; a^{2^m}= b^2=1, bab= a^{d} \rangle
$$
for $m \geq 3$ and where $d = 2^{m-1}-1$, branched over $3$ points
in $\PP^1$ with monodromy $a, b$ and $(ab)^{-1}$. Notice that such a
covering exists, since $ab(ab)^{-1} =1$. In fact, for any $m \geq 3$
there is exactly one such curve up to isomorphism. From Proposition
\ref{prop2.2} we find  the genus of $Y$,
$$
g(Y) = 2^{m-2}.
$$
Consider the curve
$$
X := Y/H
$$
where $H$ denotes the subgroup generated by $b$. We want to show
that $X$ admits complex multiplication.

As in Section 2.3 let $\xi = \xi_{2^m}$ denote a primitive $2^m$-th
root of unity. The complex  representation $V = V_1$ has character field
$K_V = \QQ(\xi + \xi^d)$, and its Schur index is equal to $1$. The
field $K_V$ is of CM-type and degree $[K_V:\QQ] = 2^{m-2}$.

\begin{prop} \label{prop4.1}
The curve $X$ has complex multiplication by $K_V$. In particular,
$$g(X) = 2^{m-3}.
$$
\end{prop}

\begin{proof}
As in Section 2.3, let $W_1$ denote the rational irreducible
representation whose complexification is $\oplus_{j=1}^{2^{m-2}}
V_1^j$. According to \cite[p.202, Corollaire]{g} and Proposition
\ref{prop2.2} above, there is an isomorphism of $\QQ[H\backslash
G/H]$-modules
\begin{equation} \label{eqn4.1}
H^1(X,\QQ) \simeq W_1^H.
\end{equation}
Since $\dim V_1^H = 1$, this implies $\dim W_1^H = 2^{m-2}$ and thus
$g(X) = 2^{m-3}$. Moreover \eqref{eqn4.1} implies that the canonical
map $\QQ[H \backslash G/H] \ra \End_{\QQ}(JX)$ induces a homomorphism
$$
p: \QQ[H \backslash G/H] \ra \End (W_1^H).
$$
whose image is isomorphic to the image of $\QQ[H \backslash G/H]$ in
$\End_{\QQ}(JX)$. Now the image of $\QQ[G]$ in $\End(W_1)$ is isomorphic
to the $2 \times 2$-matrix algebra over $K_V$, and $\QQ[H \backslash G/H]$
is isomorphic to $\mathbb{Q} \oplus K_V$ (since $\dim V^H = 1$).

Hence the image of  $\QQ[H \backslash G/H]$ in $\End(W_1^H)$ is isomorphic
to $K_V$, which means that the curve $X$ admits complex multiplication by $K_V$.
\end{proof}
This proves (2) of Theorem 2 of the Introduction. Part (4) of
Theorem 2 is proven by the following proposition.

\begin{prop} \label{prop4.2}
The Jacobian of $Y$ is isogenous to the second power of the Jacobian of X:
$$
JY \sim JX^2.
$$
In particular $JY$ admits complex multiplication by the cyclotomic field $\QQ(\xi_{2^m})$.
\end{prop}

\begin{proof}
We already know from Proposition \ref{prop2.2}  that there is only one
nontrivial isogeny factor of $JY$: it is associated to the representation
$W_1$. Using \cite[Proposition 5.2]{cr} this means that there is an
abelian subvariety $B_1$ of $JY$ such that
$$
JY \sim B_1^2,
$$
using that $\dim V_1 = 2$ and the Schur index of $V_1$ is 1. Since
$\dim V_1^H = 1$ we get moreover from \cite{cr} that
$$
JX \sim B_1.
$$
The two isogenies together imply the first assertion. Since $JX$
admits complex multiplication by a CM-subfield of index $2$ in
$\QQ(\xi_{2^m})$, $JY \sim JX^2$ admits complex multiplication by
$\QQ(\xi_{2^m})$.
\end{proof}

The projection map $\pi_b: Y \ra X$ is a double covering ramified at
two points. Hence its Prym variety $P = \ker (JY \ra JX)_0$ is a
principally polarized abelian variety, of dimension equal to $g(X) =
2^{m-3}$.

\begin{cor}
The Prym variety of the covering $\pi_b:Y \ra X$ has complex
multiplication by $K_V$.
\end{cor}
\begin{proof}
This follows immediately from Propositions \ref{prop4.1} and
\ref{prop4.2}, and the fact that $JY$ is isogenous to the product
$JX \times P$.
\end{proof}

\subsection{The CM-types of $JY$ and $JX$}

In order to determine the CM-types of $JY$ and $JX$, recall that,
according to Proposition \ref{prop2.2}, the representation
$H^1(Y,\QQ)$ of $G_m$ is just the rational irreducible
representation $W_1$. Moreover, the complexification of $W_1$ is the
direct sum of the $[K_V:\QQ] = 2^{m-2}$ Galois conjugate
representations $V_1^j$ of the complex irreducible representation
$V_1$. For every positive integer $i$ consider the complex
irreducible representation $U_i$ defined by
$$
U_i(a) = \left(
      \begin{array}{cc}
        \xi^{2i-1} & 0 \\
        0 & \xi^{(2i-1)d} \\
      \end{array}
    \right),
\qquad \ U_i(b) = \left(
      \begin{array}{cc}
        0 & 1 \\
        1 & 0 \\
      \end{array}
    \right).
$$
and denote
$$
 U_i' := U_{2^{m-2} + i}.
$$
\begin{lem} \label{lem4.3}
The complex irreducible representations $V_1^j,\; j=1, \ldots ,
2^{m-2}$ are given by the representations $U_1, \ldots, U_{2^{m-3}}$
and  $U_1', \ldots, U_{2^{m-3}}'$. The representation $U_i'$ is the
complex conjugate of $U_i$ for $i = 1, \ldots , 2^{m-3}$.
\end{lem}

\begin{proof}
First note that $U_1$ coincides with the representation $V_1$ and
every $U_i$ is Galois conjugate to $U_1$. Moreover clearly $U_1,
\ldots, U_{2^{m-3}}$ are pairwise non-isomorphic. Hence it suffices
to show that $U_i'$ is the complex conjugate of $U_i$. But this
follows from the congruences
$$
2^{m-1} + 2i -1 \equiv -(2i-1)d \mod 2^m \; \text{ and } \; (2^{m-1}
+ 2i -1)d \equiv -(2i-1) \mod 2^m.
$$
\end{proof}

\begin{prop} \label{prop4.4}
$$
H^0(Y,\omega_Y) = \oplus_{i=1}^{2^{m-3}} U_i.
$$
\end{prop}
\begin{proof}
Since $2\cdot 2^{m-3} = g(Y)$, it follows from Lemma \ref{lem4.3}
that it suffices to show that every $U_i$ with $1 \leq i \leq
2^{m-3}$ occurs exactly once in the representation $H^0(Y,\omega_Y)$
of $G_m$. Again this is a consequence of the Theorem of
Chevalley-Weil; that is, equation \eqref{eqn3.5}.

Here $Y \ra \PP^1$ is branched over $3$ points in $\PP^1$ with $n_1
= 2^m$, $n_2 = 2$ and $n_3 = 4$, and we have for the representation
$U_i, \; 1 \leq i \leq 2^{m-3}$,
\begin{eqnarray*}
N &= & -2 + \left\langle -\frac{(2i-1)}{2^m} \right\rangle +
\left\langle -\frac{(2i-1)d}{2^m} \right\rangle + \left\langle
-\frac{1}{2} \right\rangle
    + \left\langle -\frac{1}{4} \right\rangle + \left\langle -\frac{3}{4} \right\rangle\\
    & = & -2 + \frac{2^m - 2i+1}{2^m} + \frac{2^{m-1} +2i-1}{2^m} + \frac{1}{2} + \frac{3}{4} + \frac{1}{4}\\
    & = & 1.
\end{eqnarray*}
\end{proof}

As a consequence we get
\begin{cor} \label{cor4.5}
Let $\xi = e^{\frac{2\pi i}{2^m}}$. Then we have

\noindent {\em (1):} The CM-type of $JY$ is given by $\{ \xi^{2i-1}, \;|\; i=1, \ldots, 2^{m-2} \}$;\\
{\em (2):} The CM-type of $JX$ is given by $\{ \xi^{2i-1} +
\xi^{(2i-1)d} \;|\; i=1, \ldots, 2^{m-3} \}$.
\end{cor}

\begin{proof}
According to Proposition \ref{prop4.4} the CM-type of $JY$ is $\{
\xi^{2i-1},\; \xi^{(2i-1)d} \;|\; i=1, \ldots, 2^{m-3} \}$. This
implies (1), since $\xi^{(2i-1)d} \equiv 2^{m-1} -2i+1 \mod 2^m$ for
$i= 1, \ldots, 2^{m-3}$. (2) is an immediate consequence of (1),
since the CM-field of $X$ is the fixed field of the involution $\xi
\mapsto \xi^d$.
\end{proof}

\begin{prop}
The Jacobian $JX$ and the Prym variety $P$ of the covering $Y \ra X$
are simple abelian varieties of dimension $2^{m-3}$.
\end{prop}

This proves part (4) of Theorem 2 of the introduction.
\begin{proof}
It suffices to show that the field $K_V = \QQ(\xi+\xi^d)$ does not
admit a proper CM-subfield, since then every CM-type of it is
primitive, and in particular $JX$ is a simple abelian variety.

It is well known that the Galois group of the extension $\QQ(\xi)|\QQ$ is
$$
\langle \sigma \rangle \times \langle \tau \rangle ,
$$
with $\langle \sigma \rangle$ is cyclic of order $2^{m-2}$ generated
by $\sigma: \xi \mapsto \xi^5$, and $\tau$ is the involution $\tau:
\xi \mapsto \xi^{d}$. Therefore the Galois group of $K_V|\QQ$, which
is the fixed field of $\tau$, is $\langle \sigma \rangle \simeq
\ZZ/2^{m-2} \ZZ$. Its only element of order $2$ is
$\sigma^{2^{m-3}}$, which must be complex conjugation in $K_V$.
Since every nontrivial subgroup of $\langle \sigma \rangle$ contains
$\sigma^{2^{m-3}}$, this implies that every proper subfield of $K_V$
is real.
\end{proof}

\subsection{Equations}

Let the notation be as in the previous subsections. Here we want to
give equations for the curves $Y$ and $X$.

First note that the center $Z = Z(G_m) = \langle a^{2^{m-1}}
\rangle$ of the group $G_m$ is of order two; furthermore,
$$
  |Z\backslash G_m / \langle a \rangle |  = 2, \quad
|Z\backslash G_m / \langle b \rangle |  = 2^{m-1},\quad
|Z\backslash G_m / \langle ab \rangle |  = 2^{m-1}.
$$
According to \cite{r}, for any subgroup $H$ of $G$ acting on a curve
$Y$ with monodromy $g_1 , \ldots , g_t$, the
genus of the quotient $Y/H$ is given by
$$
g_{Y/H} = [G:H](g_{Y/G}-1)  +  1  + \frac{1}{2} \;
\sum_{j=1}^t \left( [G:H]-|H \backslash G /\langle g_j\rangle| \right) \ .
$$
We obtain that in our case $g(Y/Z) = 0$, and therefore $Y$ is
hyperelliptic.

We may choose coordinates in such a way that an affine equation for
$Y$ is
$$
 y^2 = x(x^{2^{m-1}} -1),
$$
and if $\xi =\xi_{2^m}$, then the automorphisms $a$ and $b$ of the
curve $Y$ are given by
$$
a(x,y) = (\xi^2 \:x,\; \xi \:y) \; \; \; ; \; \;   b(x,y) = \left( \frac{1}{\xi^2 \:x},\; \;
-i\:\xi^d\:\displaystyle\frac{y}{x^{2^{(m-2)}+1}}\right) \, .
$$

Note that
$$
a^{2^{m-1}}(x,y) =(x,-y) =: j(x,y),
$$
where $j$ denotes the hyperelliptic involution of $Y$,

Also, $b$ and $j$ generate a Klein
group, with the following associated diagram of covers.

$$ \xymatrix{
   & Y \ar[d]_{\pi_b} \ar[dr]^{\pi_{b  j}} \ar[dl]_{\pi_j} &  \\
  Y/{\langle j \rangle} =\mathbb{P}^1 \ar[dr]_{}  &  Y/{\langle b \rangle}
  = X \ar[d]_{}  & Y/{\langle b \, j \rangle} \ar[dl]^{} \\
   & \PP^1 &    }
$$
$\pi_j : Y \to  \mathbb{P}^1$ is the hyperelliptic covering ramified
over $\{ 0, \infty , \xi_{2^{m-1}}^i : 0 \leq i \leq 2^{m-1}-1\}$,
$\pi_b$ ramifies at the two fixed points $P_i = (\xi^d,y)$ with $y^2
= -2\xi^d$ of $b$, and $\pi_{bj}$ ramifies at the two fixed points
$(-\xi^d,y)$ with $y^2 = 2\xi^d$ of $bj$.

Since $X$ is then hyperelliptic and ramifies over $\pi_b(P_1)$,
$\pi_b(P_2)$, and the images under $\pi_b$ of the Weierstrass points
of $Y$, we may consider $f : \mathbb{P}^1 \to \mathbb{P}^1$ of
degree two and invariant under $x\to \frac{1}{\xi^2 \:x}$ such as
$$
f(x) = \frac {\xi\,(\xi\,x + 1)^{2}}{\xi^{2}\,x^{2} + 1}
$$
and adequate $g(x,y)$ so that $\pi_b(x,y) = (f(x), g(x,y))( =
(u,v))$, and we obtain

$$
X = Y/{\langle b \rangle} : v^2 = u(u-\xi) \prod_{k=0}^{2^{m-2}-1}
(u-f(\xi^{2k})).
$$

\begin{rem}
In the case $m =4$ we obtain the curve of genus two
$$
X = Y/{\langle b \rangle} : v^2 = u^6-5\xi_{16} u^5+2\xi_{16}^2
u^4+14\xi_{16}^3u^3-11\xi_{16}^4u^2-\xi_{16}^5u
$$
whose Jacobian has complex multiplication by $\mathbb{Q}(\xi_{16} +
\xi_{16}^7) = \mathbb{Q}(\sqrt{-2+\sqrt{2}})$.

Thus we generalize an example given in \cite{vw}, where it is given
by the equation
$$
y^2 = -x^5+3x^4+2x^3-6x^2-3x+1
$$
The curves are isomorphic, because they have the same Igusa
invariants

$i_1 := I_2^5/I_{10} = 1836660096 = 2^7\cdot 3^{15}$,

$ i_2 := I_2^3 \cdot I_4/I_{10} = 28343520 = 2^5 \cdot 3^{11} \cdot
5$, and

 $i_3 := I_2^2 \cdot I_6/I_{10} = 9762768 = 2^4 \cdot 3^9 \cdot 31$.
\end{rem}

\end{document}